\documentclass[12pt]{article}

\usepackage[dvips=false,pdftex=false,vtex=false,paperwidth=21.02cm,paperheight=29.73cm,margin=3.5cm,bottom=3cm,top=3cm,head=2cm,foot=1cm,columnsep=.5cm]{geometry}

\usepackage{graphicx}%
\usepackage{amsmath,amsfonts}%

\usepackage{amssymb}
\usepackage{fourier}
  \renewcommand{\pi}{\otherpi}
\usepackage[normal,sc,font=small]{caption}
\captionmargin \parindent

\newcommand{\T}{\textsc{t}}
\newcommand{\cM}{\mathcal{M}}

\usepackage{textcomp,listings}
\usepackage[svgnames]{xcolor} 
\definecolor{light-gray}{gray}{.95}
\definecolor{dark-gray}{gray}{0.55}
\definecolor{darker-gray}{gray}{0.3}
\usepackage[scaled=.815]{beramono} 
\def\lstbasicfont{\fontfamily{fvm}\selectfont} 
\newcommand{\mcode}[1]{\lstinline[basicstyle=\lstbasicfont,columns=fixed,basewidth=0.6em]!#1!}

\newcommand{\mi}{\hspace{.4pt}\mathrm{i}}
\newcommand{\e}{\hspace{.4pt}\mathrm{e}}

\usepackage{amsthm}
\newtheoremstyle{sftheorem}
  {}
  {}
  {\itshape}
  {}
  {\bfseries}
  {.}
  {.7em}
  {#1 \thmnumber{#2}\thmnote{~(\bfseries #3)}}

\newtheoremstyle{sfdefinition}
  {}%
  {}%
  {\normalfont}%
  {}%
  {\bfseries}%
  {.}%
  {.7em}%
  {#1 \thmnumber{#2}\thmnote{~(\bfseries #3)}}%

\theoremstyle{sfdefinition}
\newtheorem{lemma}{Lemma}[section]
\newtheorem{theorem}[lemma]{Theorem}

\theoremstyle{sfdefinition}
\newtheorem{example}[lemma]{Example}

\newcommand{\rood}[1]{\textcolor{red}{#1}}

\DeclareMathOperator{\tr}{tr}

\makeatletter 
\xdef\@endgadget#1{{\unskip\nobreak\hfil\penalty50\hskip1em\hbox{}\nobreak 
    \hfil#1\parfillskip=0pt\finalhyphendemerits=0\par}}
\def\@Endofsymbol{\decosix}
\def\Endoftheorem{\@endgadget{\@Endofsymbol}}
\makeatother

\newcommand{\mymatrix}[1]{\begin{bmatrix} #1 \end{bmatrix}}
\newcommand{\smat}[1]{\bigl[\begin{smallmatrix} #1 
                            \end{smallmatrix}\bigr]} 
\newcommand{\ssmat}[1]{\begin{smallmatrix} #1 \end{smallmatrix}}

\newcommand{\emphx}[1]{\underline{\smash{#1}}}

\newcommand{\one}{\mathbb{1}}
\newcommand{\mR}{\mathbb{R}}
\newcommand{\mC}{\mathbb{C}}
\DeclareMathOperator{\diag}{diag}
\newcommand{\ep}{\odot}

\DeclareMathOperator{\coloneq}{\,\raisebox{.08ex}{\ensuremath{:}}\!\!=}

\begin{document}

\title{Planar Pendulums in Complex Coordinates}

\author{Gjerrit Meinsma \\
  University of Twente, The Netherlands\\
  \texttt{g.meinsma@utwente.nl}}

\maketitle

\paragraph{Abstract} By exploiting complex coordinates and elementwise
matrix products, a single and simple differential equation is derived
that describes the motion of every multi-arm planar pendulum.

\section{Introduction}
The Euler-Lagrange equation applied to pendulums with multiple arms
leads to cluttered expressions involving numerous trigonometric
identities.  It is well known that trigonometric identities are easier
to handle in complex coordinates.  By exploiting complex coordinates
and elementwise products, we derive in this note a single
comprehensive equation of motion that applies to every planar
pendulum, irrespective of its number of arms, and of how the arms are
connected.  This equation is easy to simulate (although, of course, it
remains chaotic for most pendulums).  It can also be linearized, and
its special form allows to generalize a result from~\cite{BraunM2003}
regarding eigenfrequencies of multi-arm pendulums.

In \S~\ref{sec2} we explain the idea, and, to emphasize the
simplicity, provide \textsc{matlab} code that can simulate every
pendulum.  In \S~\ref{sec3} we state the general result and prove it
using force balances.  In \S~\ref{sec4} and \S~\ref{sec5} we return to
the standard angle and derive and analyze the linearization.

\emph{Notation.} In math expressions, \emph{round} brackets are used
exclusively for function arguments, e.g.~$\sin(t-1)$. We use
\emph{square} brackets for matrices and to group expressions. For
instance, $x[t-1]$ means the product of $x$ and $t-1$. The function of
a matrix is defined as the matrix of the functions, e.g.\
$\cos(\smat{a\\ b})=\smat{\cos(a)\\ \cos(b)}$ and
$\smat{a\\ b}^2=\smat{a^2\\ b^2}$.  The $\ep$ denotes
\emphx{elementwise product} of two column vectors,
e.g.~$\smat{a\\b}\ep\smat{p\\q}=\smat{ap\\bq}$.  The function $\diag$
(for ``diagonal'') turns vectors into diagonal matrices,
e.g.~$\diag(\smat{a\\ b})=\smat{a&0\\0&b}$.  We sometimes use $\one$
to mean the column vector $\one=(1,1,\ldots,1)\in\mR^n$.


\section{Overview of the main result}\label{sec2}
\begin{figure}[htbp]
  \centering  
  \includegraphics{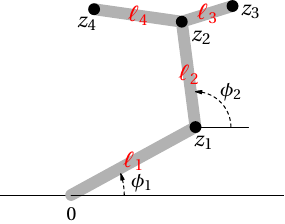}
  \caption{A pendulum with four arms (see \S~\ref{sec2}).}
  \label{fig:fun}
\end{figure}

In this section we give an overview of the main result.
Details and proofs are spelled out in the following section.

Consider the planar pendulum with four arms as shown in
Fig.~\ref{fig:fun}. The lengths of the four arms we denote by
$\ell_1,\ell_2,\ell_3,\ell_4$.  The plane $\mR^2$ we identify with the
complex plane $\mC$, and the positions of the tips of the four arms we
represent by complex numbers, $z_1,z_2,z_3,z_4\in\mC$, see
Fig.~\ref{fig:fun}.  We assume that the first pendulum is able rotate
freely around the pivot at $0\in\mC$.  Hence, the length $\ell_1$ of
the first arm equals $\ell_1=|z_1|$.  We further assume that the
masses are concentrated at the tips of the arms, and we denote the
masses by $m_1,m_2,m_3,m_4\in\mR_+$ (the black dots in the figure).
The angles of the four arms with respect to the positive real axis we
denote by $\phi_1,\phi_2,\phi_3,\phi_4\in\mR$. (Two of these are
indicated in Fig.~\ref{fig:fun}.) To formulate the
differential equation that describes the motion we need the
column vectors
\[
  z\coloneq\mymatrix{z_1\\z_2\\ z_3\\z_4},\quad 
  m\coloneq\mymatrix{m_1\\m_2\\ m_3\\m_4},\quad 
  \ell\coloneq\mymatrix{\ell_1\\\ell_2\\ \ell_3\\ \ell_4},\quad 
  \phi\coloneq
  \mymatrix{\phi_1\\ \phi_2\\ \phi_3\\ \phi_4},\quad 
  \e^{\mi\phi}=
  \mymatrix{\e^{\mi\phi_1}\\ \e^{\mi\phi_2}\\ \e^{\mi\phi_3}\\
    \e^{\mi\phi_4}}. 
\] 
The \emphx{connectivity matrix} $S$ (for lack of a better name) captures
how the arms are connected. For the pendulum of Fig.~\ref{fig:fun}
this is
\[
  \begin{array}{r|cccc|}
    \text{arm~number}:  & 1 & 2 & 3 & 4 \\ \hline 
    z_1 & 1 & 0 & 0 & 0\\
    z_2 & 1 & 1 & 0 & 0\\
    z_3 & 1 & 1 & 1 & 0\\
    z_4 & 1 & 1 & 0 & 1\\ \hline 
  \end{array}
  \quad\implies\quad S= \mymatrix{1&0&0&0\\ 1&1&0&0\\ 1&1&1&0\\1&1&0&1}. 
\]
This defines $S$.  For instance, its final row means that the tip
$z_4$ is ``arm~1 + arm~2 + arm~4'' (see Fig.~\ref{fig:fun}). The
position $z\in\mC^4$ of the tips $z$ of the four arms can now
succinctly be expressed as $S[\ell\ep\e^{\mi\phi}]$. Indeed,
\[
  z
  = \mymatrix{z_1\\ z_2\\ z_3\\ z_4}
  = \mymatrix{\ell_1\e^{\mi\phi_1}\\ z_1+\ell_2\e^{\mi\phi_2}\\
    z_2+\ell_3\e^{\mi\phi_3}\\ z_2+\ell_4\e^{\mi\phi_4}}
  = \mymatrix{1&0&0&0\\1&1&0&0\\1&1&1&0\\1&1&0&1}
  \mymatrix{\ell_1\e^{\mi\phi_1}\\ \ell_2\e^{\mi\phi_2}\\
    \ell_3\e^{\mi\phi_3}\\ \ell_4\e^{\mi\phi_4}}
  =S[\ell\ep\e^{\mi\phi}]. 
\]
In the next section we show that the equation of motion for this
pendulum is the differential equation (DE)
\begin{equation}\label{ODE}
  \ddot{\phi}
   = \cM_{\rm re}^{-1}(\phi)\big[\cM_{\rm im}(\phi)[\dot{\phi}^2]
  -g\cos(\phi)\ep\ell\ep[S^{\T}m]\bigr]. 
\end{equation}
Here, $g$ is the gravitational acceleration, and the matrices
$\cM_{\rm re}(\phi)$ and $\cM_{\rm im}(\phi)$ respectively are the
real and imaginary parts of the \emphx{complex mass matrix}
$\cM(\phi)\in\mC^{4\times 4}$ defined as
\begin{align*}
  \cM(\phi)
  & =\diag(\e^{-\mi\phi}\ep\ell)S^{\T}\diag(m)
    S\diag(\ell\ep\e^{\mi\phi}).
\end{align*}
This matrix is Hermitian and positive definite, and its real part,
$\cM_{\rm re}(\phi)$, is symmetric, positive definite and
invertible. So DE~\eqref{ODE} is well-defined. 

In the next section we show that~\eqref{ODE} is in fact the equation
of motion of \emph{every} planar pendulum, no matter how many arms it
has, and how they are connected.  All we need to do is collect the
masses and lengths, and describe the connectivity
matrix. Then~\eqref{ODE} is the DE. Stated differently, to simulate
the pendulum we only need to customize the first 6 lines of the
following \textsc{matlab} code.
\begin{lstlisting}
L = [2.0; 1.5; 0.75; 1.25];  % whatever length !\textcolor{dark-gray}{$\ell$}!, here that of Figure !\textcolor{dark-gray}{\ref{fig:fun}}! 
m = [1.0; 1.0; 1.00; 0.10];  % whatever mass !\textcolor{dark-gray}{$m$}! 
S = [1 0 0 0; 1 1 0 0; 1 1 1 0; 1 1 0 1]; % connectivity matrix 
phi0  = [0; 0; 0; 0];        % whatever initial angle !\textcolor{dark-gray}{$\phi(0)$}! 
phid0 = [0; 0; 0; 0];        % whatever initial angular velocity !\textcolor{dark-gray}{\smash{$\dot{\phi}(0)$}}! 
t = 0:.01:5;                 % if we want to simulate 5 seconds 
[phi,phid] = pendusim(L,m,S,phi0,phid0,t); % do the simulation 
\end{lstlisting}
It calls \mcode{pendusim} to simulate $\phi(t),\dot{\phi}(t)$, and
this works for every pendulum:
\begin{lstlisting}
function [phi,phid] = pendusim(L,m,S,phi0,phid0,t) 
LSMSL = diag(L)*S'*diag(m)*S*diag(L); 
n = length(L); 
g = 9.81;                    % assuming we are on earth                  
MM = @(phi)diag(exp(-i*phi))*LSMSL*diag(exp(i*phi)); % !\textcolor{dark-gray}{$\cM(\phi)$}! 
xdot = @(t,x,M) [x(n+1:end); % derivative of !\textcolor{dark-gray}{$x(t)\coloneq[\phi(t); \smash{\dot{\phi}(t)}]$}! 
     real(M)\(imag(M))*(x(n+1:end).^2)-g*cos(x(1:n)).*L.*(S'*m)]; 
[~,x] = ode23(@(t,x)xdot(t,x,MM(x(1:n))),t,[phi0;phid0]); 
phi = x(:,1:n)';             % the columns are the !\textcolor{dark-gray}{$\phi(t)$}! 
phid = x(:,n+1:n)';          % the columns are the !\textcolor{dark-gray}{\smash{$\dot{\phi}(t)$}}! 
\end{lstlisting}


\section{Force Balances in $\mC$}\label{sec3}
\begin{figure}[htbp]
  \centering
  \includegraphics{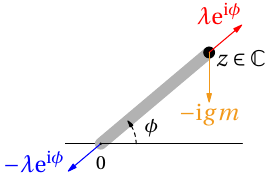}
  \caption{Single-arm pendulum in $\mC$ (see 
    Example~\ref{SP}). The compression forces are shown in red and
    blue. The gravitational force is in shown in orange.}
  \label{fig:SDP}
\end{figure}
In this section we derive DE~\eqref{ODE} using force balances, and
we show that it applies to all planar pendulums.   First
the simplest case:
\begin{example}[Single-arm pendulum]\label{SP} Consider the
  single-arm pendulum in $\mC$ as depicted in Fig.~\ref{fig:SDP}. The
  pendulum (its one arm) can rotate freely around the pivot at
  $0\in\mC$.  Let $\ell$ be the length of the arm, and let $z\in\mC$
  be the position of the tip of the pendulum. Hence,
  \[
    z=\ell\e^{\mi\phi},
  \]
  where ``$\mi$'' is the imaginary unit, and $\phi$ is the argument of the
  complex number $z$. The speed and acceleration,
  $\dot{z},\ddot{z}\in\mC$, thus are
  \begin{align*}
    \dot{z} 
    = \ell\e^{\mi\phi}\mi\dot{\phi},\qquad
    \ddot{z}
    = \ell\e^{\mi\phi}[\mi\ddot{\phi}-\dot{\phi}^2].
  \end{align*}
  The force $F\in\mC$ acting on the tip of the arm (at $z$) consists
  of gravity $gm$ (downwards, so in the ``$-\mi$ direction''), and a
  compression force $\lambda$ in the direction of the pendulum, see
  Fig.~\ref{fig:SDP}. Hence the total force acting on the tip is
  $-\mi g m +\lambda \e^{\mi\phi}\in\mC$ for some as yet unknown
  compression force $\lambda\in\mR$. Newton's 2nd law, $m\ddot{z}=F$,
  therefore becomes
  \[
    m\ell\e^{\mi\phi}[\mi\ddot{\phi}-\dot{\phi}^2]
    = -\mi g m +\lambda\e^{\mi\phi}. 
  \]
  Dividing by $m\ell\e^{\mi\phi}$ gives
  \[
    \mi\ddot{\phi}-\dot{\phi}^2 = -\mi
    \frac{g}{\ell}\e^{-\mi\phi}+\frac{1}{m\ell}\lambda.
  \]
  This equation is linear in the real pair $(\ddot{\phi},\lambda)$,
  and its imaginary and real part respectively determine $\ddot{\phi}$
  and $\lambda$:
  \begin{align*}
    \ddot{\phi} &=-\frac{g}{\ell}\cos(\phi),\\
    \lambda & = -m\ell \dot{\phi}^2+gm \sin(\phi). 
  \end{align*}
  The first of these is the equation of motion. Once we have
  $\phi,\dot{\phi}$, the second equation determines the compression
  force $\lambda$.
  \Endoftheorem
\end{example}
Now the general case. We consider arbitrary \emphx{pendulum trees} ---
like trees in graph theory --- by which we mean:
\begin{itemize}
\item the pendulum consists of a finite number of arms,
\item the arms are connected at the tips (the joints) of the arms,
\item the pendulum contains no cycles,
\item the tip of one-and-only-one arm is fixed to the pivot located at
  $0\in\mC$ (but the arm can rotate freely around the pivot),
\item all arms can rotated freely around their joints (as in, there is
  no friction), and there is a gravitational force,
\item the mass of every arm is concentrated at the tips of the
  arm.
\end{itemize}

For general pendulum trees the derivation of the equation of motion
follows almost the same steps as that of the single-arm pendulum,
except that now we need to describe the connectivity. Let $n$ be
the number of arms, and introduce the vectors of positions, masses,
lengths, angles and exponentials:
\[
  z=\mymatrix{z_1\\ \vdots \\z_n}, ~
  m=\mymatrix{m_1\\ \vdots\\m_n}, ~
  \ell=\mymatrix{\ell_1\\ \vdots\\ \ell_n}, ~
  \phi=\mymatrix{\phi_1\\ \vdots \\ \phi_n}, ~
  \e^{\mi\phi}=\mymatrix{\e^{\mi\phi_1}\\ \vdots\\ \e^{\mi\phi_n}}.
\]
In Lemma~\ref{lem:dual} we show that
\begin{equation}
  \label{eq:cm}
  z= S[\ell\ep\e^{\mi\phi}]
\end{equation}
for some $n\times n$ logical matrix $S$. This $S$ we call the
\emphx{connectivity matrix}.  The first and second derivative of $z$
with respect to time are
\[
  \dot{z}=S\bigl[\ell\ep\e^{\mi\phi}\ep[\mi\dot{\phi}]\bigr],
  \qquad
  \ddot{z}=S\bigl[\ell\ep\e^{\mi\phi}\ep[\mi\ddot{\phi}-\dot{\phi}^2]\bigr].
\]
Each joint $k\in\{1,\ldots,n\}$ experiences a gravitational force,
$-\mi gm_k$, and a number of compression forces, one for each arm that
is connected to that joint.  Let $F_k\in\mC$ be the total force acting
on joint $k$, and let $F\in\mC^n$ be the vector of all $n$ total
forces acting on the $n$ joints/tips. In Lemma~\ref{lem:dual} we show
that
\begin{equation}\label{fizzz}
  F=-\mi g m+ B[\lambda\ep\e^{\mi\phi}]
\end{equation}
where $B$ is some $n\times n$ matrix, and $\lambda\in\mR^n$ are the
$n$ compression forces ``in'' the $n$ arms. The term $-\mi gm$ is the
vector of gravitational forces, and $B[\lambda\ep\e^{\mi\phi}]$ is
the vector of forces due to compression.
\begin{lemma}[Duality]\label{lem:dual}
  For every pendulum tree with $n$ arms we have that~\eqref{eq:cm}
  and~\eqref{fizzz} hold for some $S\in\{0,1\}^{n\times n}$ and
  $B\in\{-1,0,1\}^{n\times n}$. In fact, $B=S^{-\T}$, and all
  eigenvalues of $S$ and $B$ are equal to 1.
\end{lemma}
\begin{proof}
  We prove it by induction in the number of arms $n$. For $n=1$ arms
  (the single-arm pendulum) we have $S=B=1$ (see Example~\ref{SP}) so
  then the result holds.

  Next let $n\in\mathbb{N}$ and suppose that $B=S^{-\T}$ for every
  pendulum tree of $n$ arms, and that all eigenvalues of $S$ and $B$
  are $1$, and $S\in\{0,1\}^{n\times n}$ and
  $B\in\{-1,0,1\}^{n\times n}$ (the induction hypothesis).  Now
  consider an arbitrary pendulum tree with $n+1$ arms. By removing one
  arm it reduces to an $n$-arm pendulum. Without loss of generality we
  denote the tip of the removed arm by $z_{n+1}$, and the tip to which
  it was connected by $z_n$, see Fig.~\ref{fig:fant}. Thus,
  $z_{n+1}=z_n+\ell_{n+1}\e^{\mi\phi_{n+1}}$. The connectivity matrix
  $S_{n+1}$ of the $n+1$-arm pendulum in terms of the connectivity
  matrix $S$ of the $n$-arm pendulum follows from
  \[
    \mymatrix{z\\ z_{n+1}}
    =
    \mymatrix{z\\ z_n + \ell_{n+1}\e^{\mi\phi_{n+1}}}
    = \underbrace{\mymatrix{S & 0\\ S_{n*} &
        1}}_{S_{n+1}}\mymatrix{\ell\e^{\mi\phi}\\ \ell_{n+1}\e^{\mi\phi_{n+1}}},
  \]
  where $S_{n*}$ means the final row of $S$. By adding arm number
  $n+1$ to the pendulum, the total force $\smat{F\\ F_{n+1}}$ (acting
  on the $n+1$ joints/tips) becomes
  \[
    \mymatrix{F\\ F_{n+1}}
    = -\mi g \mymatrix{m\\ m_{n+1}}
    +\underbrace{\left[\begin{array}{c|c}
                         B & \ssmat{0\\[-1ex] \vdots\\ 0\\-1}\\ \hline
                         0 & 1
                       \end{array}\right]}_{B_{n+1}}
                   \mymatrix{\lambda\e^{\mi\phi}\\
                     \lambda_{n+1}\e^{\mi\phi_{n+1}}},
  \]
  (this defines $B_{n+1}$).  This equation expresses that there is one
  additional compression force $\lambda_{n+1}\e^{\mi\phi_{n+1}}$ and
  that it acts on $z_{n+1}$, and also on $z_n$ but then in opposite direction
  (see the red and blue forces in Fig.~\ref{fig:fant}).  By the
  induction hypothesis we have $B=S^{-\T}$, so
  \[
    S_{n+1}^{\T}B_{n+1}=
    \left[\begin{array}{c|c} S^{\T} & S_{n*}^{\T} \\ \hline 0 & 1\end{array}\right]
    \left[\begin{array}{c|c} S^{-\T} &
      \ssmat{0\\[-1ex] \vdots\\ 0\\-1}\\ \hline 0 & 1\end{array}\right]
      =\left[\begin{array}{c|c} I & 0 \\
           \hline 0 & 1\end{array}\right]=I_{n+1}.
  \]
  Hence, $B_{n+1}=S_{n+1}^{-\T}$, and the eigenvalues of $S_{n+1}$ are
  those of $S$ with one more eigenvalue equal to 1. So, by the
  induction hypothesis, they are all equal to $1$. Also, we see that
  $S\in\{0,1\}^{n\times n}$ and $B\in\{-1,0,1\}^{n\times n}$.
\end{proof}
\begin{figure}
  \centering
  \includegraphics{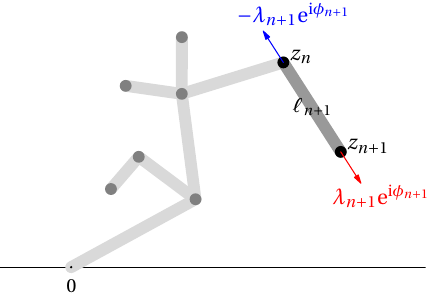}
  \caption{An arbitrary pendulum with $n+1$ arms, seen as a pendulum
    with $n$ arms (light gray) with one more arm added (dark gray).
    The compression force in the added arm is indicated in red and blue.}
  \label{fig:fant}
\end{figure}
The vector of forces $F\in\mC^n$ as given in~\eqref{fizzz} thus equals
$-\mi g m+S^{-\T}[\lambda\ep\e^{\mi\phi}]$, and, so, Newton's 2nd law,
$m\ep\ddot{z}=F$, becomes
\begin{equation}\label{xyzzz}
  m \ep \biggl[S \bigl[\ell\ep\e^{\mi\phi}
  \ep[\mi\ddot{\phi}-\dot{\phi}^2]\bigr]\biggr]
  = -\mi g m+S^{-\T}[\lambda\ep\e^{\mi\phi}]. 
\end{equation}
This equation settles both the equation of motion and the compression
forces:
\begin{theorem}[Complex DE for pendulum trees in $\mC$]\label{CNFP}
  Suppose we have a pendulum tree with $n$ arms and connectivity
  matrix $S$, mass vector $m\in\mR_+^n$, vector of angles
  $\phi:\mR\to\mR^n$, and vector of lengths
  $\ell\in\mR^n_+$. 
  Then
  \begin{align}\label{aaa}
    \ddot{\phi}
    & = \cM_{\rm re}^{-1}(\phi)\big[\cM_{\rm im}(\phi)\dot{\phi}^2
      -g\cos(\phi)\ep\ell\ep[S^{\T}m]\bigr],\\
    \lambda
    & =g\sin(\phi)\ep[S^{\T}m]-\ell^{-1}\ep\bigl[\cM_{\rm 
      im}(\phi)\ddot{\phi}+\cM_{\rm re}(\phi)[\dot{\phi}^2]\bigr].
      \label{bbb}
  \end{align}
  Equation~\eqref{aaa} is the \emphx{real equation of motion},
  and~\eqref{bbb} determines the compression forces $\lambda$.
  Here, $\cM_{\rm re}(\phi)$ and $\cM_{\rm im}(\phi)$ are the real and
  imaginary parts of the $n\times n$ \emphx{complex mass matrix}
  $\cM(\phi)$ defined as
  \begin{equation}\label{massmat}
    \cM(\phi)=\diag(\e^{-\mi\phi}\ep\ell)
    S^{\T}\diag(m) S\diag(\ell\ep\e^{\mi\phi}).
  \end{equation}
  This complex mass matrix is Hermitian and positive definite, and
  its real part is symmetric and positive definite and invertible.
\end{theorem}
\begin{proof}
  First, realize that $\cM(\phi)$ defined in~\eqref{massmat} is
  Hermitian and positive definite.  Thus its real part is symmetric
  and positive definite. Positive definite matrices are invertible.
  
  The elementwise product $a\ep b$ of two column vectors $a,b$ is the
  same as the usual matrix-vector product $\diag(a)b$.
  So~\eqref{xyzzz} is the same as
  \[
    \diag(m)S\diag(\ell\ep\e^{\mi\phi})[\mi\ddot{\phi}-\dot{\phi}^2]=
      -\mi g m + S^{-\T}[\lambda\ep\e^{\mi\phi}]. 
  \]
  Premultiplying by the invertible $\diag(\e^{-\mi\phi}\ep\ell)S^{\T}$
  turns this into
  \begin{equation}\label{ooo}
    \cM(\phi) [\mi\ddot{\phi}-\dot{\phi}^2]
    =
    -\mi g\e^{-\mi\phi}\ep\ell\ep[S^{\T}m]+ \ell\ep\lambda. 
  \end{equation}
  We call this the \emphx{complex equation of motion}. The imaginary
  part of this expression is
  \begin{equation}\label{ODEn}
    \cM_{\rm re}(\phi)\ddot{\phi}-\cM_{\rm 
      im}(\phi)\dot{\phi}^2=-g\cos(\phi)\ep\ell\ep[S^{\T}m]
  \end{equation}
  which determines the real equation of motion~\eqref{aaa}. Once
  $\ddot{\phi}$ is known, the real part of~\eqref{ooo} uniquely
  determines the compression forces $\lambda$ as given in~\eqref{bbb}.
\end{proof}

For simulation we only need~\eqref{aaa}, i.e.~we can discard the
equation for the compression forces~\eqref{bbb}. 

Modelling pendulums thus boils down to determining the connectivity
matrix $S$. As we saw in \S~\ref{sec2}, this matrix follows easily
from the configuration. For example, the standard triple pendulum has
a lower triangular connectivity matrix with ones on and below the
diagonal:
\[
  \vcenter{\hbox{\includegraphics{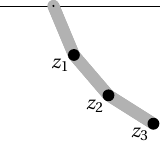}}}
  \quad \implies\quad  
  \mymatrix{z_1\\ z_2\\ z_3}
  =
  \mymatrix{1&0&0\\ 1&1&0\\ 1&1&1}
  \mymatrix{"\text{arm}_1"\\"\text{arm}_2"\\"\text{arm}_3"}
  \quad \implies\quad  
  S=\mymatrix{1&0&0\\ 1&1&0 \\ 1&1&1},  
\] 
and for the following \emph{non}standard triple pendulum we have
\[
  \vcenter{\hbox{\includegraphics{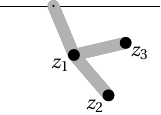}}}
  \quad \implies\quad 
  \mymatrix{z_1\\ z_2\\ z_3}
  =
  \mymatrix{1&0&0\\ 1&1&0\\ 1&\rood{0}&1}
  \mymatrix{"\text{arm}_1"\\"\text{arm}_2"\\"\text{arm}_3"}
  \quad \implies\quad 
  S=\mymatrix{1&0&0\\ 1&1&0 \\ 1&\rood{0}&1}.
\]


\section{Standard angle}\label{sec4}
In complex analysis one always defines the angle $\phi$ relative to
the positive real axis. For pendula it is common to define the angle
relative to the hanging position (we mean relative to
$\phi=-\pi/2$). Let $\theta$ be this angle:
\[
  \theta\coloneq \phi+\frac12\pi\one \in\mR^n,
\]
where $\one=(1,1,\ldots)\in\mR^n$. It is easy to see that
$\cM(\theta)=\cM(\phi)$.  (In fact $\cM(\phi+c\one)=\cM(\phi)$ for
every $c\in\mR$.)  So the complex equation of motion~\eqref{ooo} in
terms of $\theta$ is
\[
  \cM(\theta) [\mi\ddot{\theta}-\dot{\theta}^2]
  -g\e^{-\mi\theta}\ep\ell\ep[S^{\T}m]- \ell\ep\lambda = 0,
\]
and from its imaginary part the real equation of motion follows,
\begin{equation}\label{ODEtheta}
  \cM_{\rm re}(\theta)\ddot{\theta}-\cM_{\rm 
    im}(\theta)\dot{\theta}^2+g\ell\ep[S^{\T}m]\ep\sin(\theta)=0. 
\end{equation}
It is perhaps good to \emph{once} work out~\eqref{ODEtheta} for a
special case, and see if it recovers the standard but gruesome real
formula's.  We do this for the double pendulum, where it is still
manageable.
\begin{example}[Double pendulum --- from complex to real]
  Consider the double pendulum, see Fig.~\ref{fig:dp}. Now
  $S=\smat{1&0\\1&1}$ so the complex mass matrix~\eqref{massmat} in
  this case is
  {\arraycolsep 3pt%
    \begin{align*}
    \cM(\theta) 
    & = \mymatrix{\ell_1\e^{-\mi\theta_1}&0\\0&\ell_2\e^{-\mi\theta_2}}
        \mymatrix{1&1\\ 0&1}\mymatrix{m_1&0\\0&m_2}\mymatrix{1&0\\1&1}
        \mymatrix{\ell_1\e^{\mi\theta_1}&0\\0&\ell_2\e^{\mi\theta_2}}\\
     & = \mymatrix{\ell_1^2(m_1+m_2) & \ell_1\ell_2 m_2 \e^{-\mi(\theta_1-\theta_2)}\\
        \ell_1\ell_2m_2\e^{\mi(\theta_1-\theta_2)} & \ell_2^2 m_2}. 
  \end{align*}}%
  Its real and imaginary parts are 
  \[
    \cM_{\rm re}(\theta)=
    \mymatrix{\ell_1^2(m_1+m_2) & \ell_1\ell_2 m_2 \cos(\theta_1-\theta_2)\\
      \ell_1\ell_2m_2\cos(\theta_1-\theta_2) & \ell_2^2 m_2}, 
  \]
  \[
    \cM_{\rm im}(\theta)=
    \mymatrix{0 & -\ell_1\ell_2 m_2 \sin(\theta_1-\theta_2)\\
      \ell_1\ell_2m_2\sin(\theta_1-\theta_2) & 0}. 
  \]
  With it, the equation of motion~\eqref{ODEtheta} becomes the familiar
  \begin{align*}
    &\mymatrix{
    \ell_1^2(m_1+m_2)\ddot\theta_1+\ell_1\ell_2m_2\cos(\theta_1-\theta_2)\ddot\theta_2\\
        \ell_1\ell_2m_2\cos(\theta_1-\theta_2)\ddot\theta_1+\ell_2^2m_2\ddot\theta_2} \\
     & \qquad - \mymatrix{
           -\ell_1\ell_2m_2\sin(\theta_1-\theta_2)\dot\theta_1^2\\
           +\ell_1\ell_2m_2\sin(\theta_1-\theta_2)\dot\theta_2^2}
    + g\mymatrix{\sin(\theta_1)\ell_1(m_1+m_2)\\
                   \sin(\theta_2)\ell_2m_2}=0.
  \end{align*}
  \Endoftheorem
\end{example}
For pendula with more arms the worked out equations are rather messy.
The point of this note is that we do not have to work it out to such
level of detail. It is more insightful to use~\eqref{aaa}
or~\eqref{ODEtheta}. For example, it simplifies the analysis of the
linearization. This we show in the next section.


\section{Linearization and eigenfrequencies}\label{sec5}
The linearization of the equation of motion~\eqref{ODEtheta} at the
hanging position ($\theta=0$) is
\begin{equation}\label{eq:linz}
  \cM_{\rm re}(0)\ddot\theta+g\ell\ep[S^{\T}m]\ep\theta=0, 
\end{equation}
or, equivalently,
\begin{equation}\label{eq:lin}
  \ddot\theta = -A\theta,
  \qquad
  A \coloneq  g\cM_{\rm re}^{-1}(0)\diag(\ell\ep[S^{\T}m])
  \in\mR^{n\times n}.
\end{equation}

\begin{lemma}[Complete set of eigenmodes]
  For every $n\times n$ connectivity matrix $S$, and every
  $m,\ell\in\mR_+^n,g>0$, the matrix $A$ defined in~\eqref{eq:lin} has
  $n$ positive real eigenvalues $\lambda_1,\ldots,\lambda_n\in\mR_+$,
  and $n$ linearly independent real eigenvectors
  $v_1,\ldots,v_n\in\mR^n$. Consequently, the
  linearization~\eqref{eq:lin} has $n$ complex eigenmodes
  $\theta_k(t)\coloneq v_k\e^{\mi\omega_k t}, k=1,\ldots,n$,
  where $\omega_k\coloneq\sqrt{\lambda}_k$, and every solution
  $\theta(t)$ of~\eqref{eq:lin} is (the real part of) a linear
  combination of these eigenmodes.
\end{lemma}
\begin{proof}
  Let $D=\diag(\ell\ep[S^{\T}m])$ and realize that
  $A=g\cM_{\rm re}^{-1}(0)D$.  All entries of $S^{\T}m$ are positive
  because all entries of $m$ are positive, and $S$ is invertible and
  has nonnegative entries only. Since also all entries of $\ell$ are
  positive we see that $D$ has a positive definite square root,
  $\sqrt{D}$.  Now, $A$ is isomorphic to
  $g\sqrt{D}\cM_{\rm re}^{-1}(0)\sqrt{D}$. The latter is symmetric
  positive definite, hence, $\mR^n$ has a basis
  $ v_1,\ldots,v_n$ of eigenvectors of $A$, and the
  eigenvalues $\lambda_1,\ldots,\lambda_n$ of $A$ are real and
  positive. Clearly,
  $\theta(t)\coloneq v_k\e^{\mi\omega_kt}$
  satisfies~\eqref{eq:lin} iff $\omega_k=\pm\sqrt{\lambda_k}$.
\end{proof}
The matrix $\cM_{\rm re}(0)$ equals
$\diag(\ell)S^{\T}\diag(m)S\diag(\ell)$, and therefore, the matrix $A$
defined in~\eqref{eq:lin} equals
\begin{equation}\label{AAis}
  A = g\diag(\ell^{-1})S^{-1}\diag(m^{-1})S^{-\T}\diag(S^{\T}m).
\end{equation}
This one seems to be hard to analyze, but motivated
by~\cite{BraunM2003}, the following can be shown. It generalizes
\cite[Proposition~1]{BraunM2003} to arbitrary pendulum trees:
\begin{lemma}
  For every $n\times n$ connectivity matrix $S$, and every
  $\ell\in\mR_+^n,m\in\mR^n_+,g>0$, the diagonal entries of $A^{-1}$
  equal $(A^{-1})_{kk}=\ell_k/g$. In particular, the trace of $A^{-1}$
  equals $\ell_{\rm tot}/g$ where
  $\ell_{\rm tot}\coloneq\sum_{k=1}^n\ell_k$. Hence, the
  eigenfrequencies $\omega_1,\ldots,\omega_n$ of the
  linearization~\eqref{eq:lin} satisfy
  \[
    \sum_{k=1}^n \frac1{\omega_k^2}=\frac{\ell_{\rm tot}}{g}.
  \]
\end{lemma}
\begin{proof}
  We have $A^{-1}=\diag(S^{\T}m)^{-1}S^\T\diag(m)S\diag(\ell/g)$. From
  that it follows that $(A^{-1})_{kk}$ equals $\ell_k/g$ for all
  $k=1,\ldots,n$ iff
  \begin{equation}\label{eq:gek}
    (S^{\T}m)_k = (S^{\T}\diag(m)S)_{kk}, \qquad k=1,\ldots,n.
  \end{equation}
  The latter we prove by induction. For $n=1$ (the single-arm
  pendulum) the result holds. Now let $n\in\mathbb{N}$, and
  assume~\eqref{eq:gek} holds for pendulums with $n$ arms. The
  connectivity matrix for an $n+1$-arm pendulum is
  $\smat{S&0\\ S_{n*}&1}$ (see the proof of Theorem~\ref{CNFP}). So
  for an $n+1$-arm pendulum the expression ``$S^{\T}m$'' becomes
  \[
    \mymatrix{S^{\T} & S_{n*}^{\T}\\ 0 & 1}
    \mymatrix{m\\ m_{n+1}}
    =
    \mymatrix{S^{\T}m + S_{n*}^{\T}m_{n+1}\\ m_{n+1}}. 
  \]
  We need to show that its entries equal the diagonal entries of
  \begin{align*}
    \mymatrix{S^{\T} & S_{n*}^{\T}\\ 0 & 1}
    &
    \mymatrix{\diag(m)&0\\0&m_{n+1}} \mymatrix{S & 0 \\ S_{n*} & 1}
    =
    \mymatrix{S^{\T}\diag(m)S+S_{n*}m_{n+1}S_{n*}^{\T}& S_{n*}^{\T}m_{n+1}\\
        m_{n+1}S_{n*} & m_{n+1}}.
  \end{align*}
  That is the case because of the induction hypothesis and the fact
  that all entries of $S_{n*}$ are either zero or one.

  Consequently, $\tr(A^{-1})=\ell_{\rm tot}/g$. It is a classic result
  that the trace of a matrix equals the sum of its eigenvalues, and
  the eigenvalues of $A^{-1}$ are $1/\omega_k^2$.
\end{proof}

\begin{figure}[htbp]
  \centering 
  \includegraphics[scale=1]{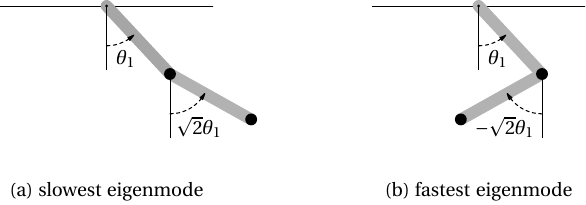}
  \caption{The two eigenmodes of the standard double pendulum. For the 
    slowest eigenmode, the second angle is always $\sqrt{2}$ times the first 
    angle. For the fastest eigenmode, the second angle is always 
    $-\sqrt{2}$ times the first angle. This assumes all masses are the 
    same and all lengths are the same. See Example~\ref{ex:dtp}.}
  \label{fig:dp}
\end{figure}
\begin{figure}[htbp]
  \centering 
  \includegraphics[scale=1]{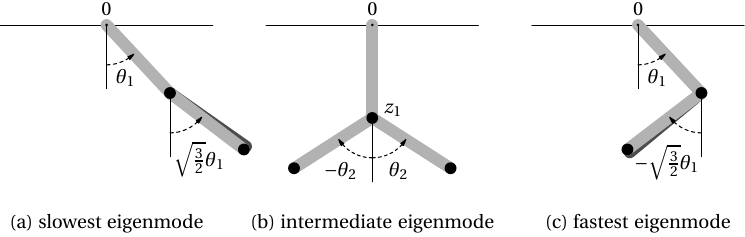}
  \caption{The three eigenmodes of a non-standard triple pendulum. For 
    the slowest mode, the second and third angle coincide and always 
    equal $\sqrt{3/2}$ times the first angle. The intermediate eigenmode is 
    where the first arm does not move and the second and third arm 
    move in opposite direction.  For the fastest eigenmode, the second and 
    third angle again conincide, and they always equal $-\sqrt{3/2}$ times 
    the first angle. This assumes all masses are the same and all 
    lengths are the same. See Example~\ref{ex:dtp}.}
  \label{fig:tp}
\end{figure}

\begin{example}[Standard double and non-standard triple pendulum]\label{ex:dtp}
  An interesting special class of pendulums is when all masses are the 
  same and all lengths are the same. Then \eqref{AAis} simplifies to 
  \[
    A=\frac{g}{\ell_1}\, S^{-1}S^{-\T}\diag(S^{\T}\one). 
  \]
  Within this class we consider two cases. 

  Consider first the standard double pendulum of Fig.~\ref{fig:dp}. 
  Now the connectivity matrix is $S=\smat{1&0\\1&1}$, and the matrix 
  $A$, the eigenfrequencies and eigenvectors turn out to be 
  \begin{align*}
    A = \frac{g}{\ell_1}\mymatrix{2&-1\\-2&2}, 
    \quad 
    \mymatrix{\omega_1\\\omega_2}
    =\sqrt{\frac{g}{\ell_1}}
       \mymatrix{\sqrt{2-\sqrt2}\\ \sqrt{2+\sqrt2}}, 
    \quad 
     v_1 =\mymatrix{1\\\sqrt2}, \quad 
     v_2 =\mymatrix{1\\-\sqrt2}. 
  \end{align*}
  Two corresponding eigenmodes $v_k\e^{\mi\omega_k t}$
  are indicated in Fig.~\ref{fig:dp}. 

  Consider next the \emph{non}-standard triple pendulum whose three 
  arms all connect at $z_1$, see Fig.~\ref{fig:tp}(b). The 
  connectivity matrix $S$, the matrix $A$ and eigenfrequencies are 
  \[
    S=\mymatrix{1&0&0\\1&1&0\\1&0&1}, 
    \quad 
    A=\frac{g}{\ell_1}\mymatrix{3&-1&-1\\-3&2&1\\-3&1&2}, 
    \quad 
    \mymatrix{\omega_1\\ \omega_2\\ \omega_3}
    =\sqrt{\frac{g}{\ell_1}} \mymatrix{\sqrt{3-\sqrt6}\\ 1\\ \sqrt{3+\sqrt6}}, 
  \]
  and corresponding eigenvectors are 
  \[
     v_1=\mymatrix{1\\ \sqrt{3/2}\\ \sqrt{3/2}},\qquad 
     v_2=\mymatrix{0\\ 1\\ -1},\qquad 
     v_3=\mymatrix{1\\ -\sqrt{3/2}\\ -\sqrt{3/2}}. 
  \]
  Figure~\ref{fig:tp} depicts the three eigenmodes 
  $v_k\e^{\mi\omega_k t}, k=1,2,3$.  \Endoftheorem 
\end{example}


\bibliography{pendulumsincc}

\begin{thebibliography}{1}

\bibitem{BraunM2003}
M.~Braun.
\newblock On some properties of the multi pendulum.
\newblock {\em Archive of Applied Mechanics}, 72:899--910, 2003.
\newblock DOI: 10.1007/s00419-002-0263-4.

\end{thebibliography}

\end{document}